\documentclass{amsart}

\usepackage[british]{babel}
\usepackage[utf8]{inputenc}

\usepackage{amsmath,amssymb,amsfonts,amsthm}

\usepackage{mathptmx}

\usepackage{graphicx}
\usepackage{float}
\usepackage{subfigure}
\usepackage[font=small,labelfont=bf]{caption}

\usepackage{xcolor}
\usepackage{mdframed}
\usepackage[many]{tcolorbox}
\tcbuselibrary{skins}

\usepackage{enumitem}
\usepackage{indentfirst}
\usepackage{multicol}

\usepackage{tikz}
\usetikzlibrary{shapes.geometric}

\usepackage[square,numbers]{natbib}

\usepackage{hyperref}

\newtheorem{theo}{Theorem}[section]
\newtheorem{defi}{Definition}[section]
\newtheorem{prop}{Proposition}[section]
\newtheorem{coro}{Corollary}[section]
\newtheorem{lema}{Lemma}[section]
\newtheorem{claim}{Claim}[section]
\newtheorem{ex}{Example}[section]

\newtheoremstyle{meuestilo}
  {6pt plus 2pt minus 2pt}
  {6pt plus 2pt minus 2pt}
  {}
  {1.5em}
  {\itshape}
  {.}
  { }
  {}

\theoremstyle{meuestilo}
\newtheorem{remark}{Remark}

\tcbset{
  myinlinebox/.style={
    enhanced,
    boxrule=0.5pt,
    colback=yellow!20,
    colframe=orange!80!black,
    sharp corners,
    boxsep=1pt,
    left=2pt,
    right=2pt,
    top=1pt,
    bottom=1pt,
    fontupper=\footnotesize,
  }
}

\newtcbox{\DD}{on line, myinlinebox}

\renewenvironment{proof}[1][Proof]%
{\noindent\textbf{#1.} }%
{\hfill\rule{0.7em}{0.7em}\par}

{\noindent\textbf{#1:} }%
{\hfill$\Box$\par}

\newmdenv[
  backgroundcolor=yellow!20,
  nobreak=true
]{paint}

\title{Geometric Obstructions in Finsler Spaces and Torsion-Free Persistent Homology}

\author{Rafael Cavalcanti}
\address{Federal University of Pernambuco, Recife, Brazil}
\email{rafael@dmat.ufpe.br}
\date{\today}
\begin{document}
\begin{abstract}
    We relate the novel concept of Topological Data Analysis in Finsler space with representability property, which is expected to prevent spurious features. We use decomposition of integer matrix in order to find suitable prime integer $p$ such that homologies over $\mathbb{Z}_p$ encompasses only the holes associated to the free part.
\end{abstract}
\maketitle

\tableofcontents
\markboth
  {Rafael Cavalcanti}
  {Geometric Obstructions in Finsler Spaces and Torsion-Free Persistent Homology.}
\newpage
\section{Introduction}
\par It is not unknown that theoretical results of simplicial complexes and knowledge of their homology groups have improved the understanding of many data set by the use of Topological Data Analysis, for short TDA. The main example of this phrase is the Nerve Theorem.
\begin{theo}[Alexandroff, \cite{alexandroff1928}]
        Let $\mathcal{U}$ be a good cover of a topological space $X$. Then, the geometric realization of $\mathcal{N(U)}$ is homotopy equivalent to $\displaystyle\bigcup_{U\in\mathcal{U}}U$.
\end{theo}
The nerve $\mathcal{N(U)}$ mentioned in the above theorem is thought as a Cech complex in real applications of TDA, as later we will in Definition \ref{cech_complex}. Although, Cech complex deals with many high order interaction depending of the amount of data. Usually TDA aims to identify relations higher than simply one dimension interactions. Thence, the following inclusion, solves partially this problematic, reducing the high interaction to pairwise relation.

\begin{prop}\label{inclusion_1}
    Let $\epsilon >0$. We have the inclusions
    \[
     Rips(X,  \epsilon) \subseteq \check{C}ech(X, \epsilon) \subseteq Rips(X,  2\epsilon).
    \]
    
\end{prop}
\par Even the main simplicial complex that rules the topology of the data is the Cech complex, we may use the above inclusion to reduce the complexity and approximate the results. When dealing with Representability property of simplicial complexes, we may infer some consequences about the homology of the addressed complexes. For instance, we will see that in a Finsler space, ta data can not form "simple holes".
\par Often, in many analysis of TDA, the choice of coefficients for homology is regarded as a minor issue, however, even if the choice $\mathbb{Z/}2\mathbb{Z}=\mathbb{Z}_2$ provides no trouble about orientation, this choice may contain noise in the form of torsion. We still do not know if torsion appears with high probability of being noise, but we propose a set of coefficients based on the Smith Normal Form of the boundary operators such that the original torsion part of the homology groups is not carried on the homologies over this new set of coefficients.
\par Our main conclusions are Corollaries~\ref{geometricdegree}
and~\ref{goodprimes}.
The former imposes a geometric constraint that rules out spurious
high-dimensional features in the filtration, while the latter establishes an
algebraic stability result: outside a finite set of primes, persistent homology
with $\mathbb{Z}_p$ coefficients reflects only free homology classes and hence
agrees with the rank invariants over $\mathbb{Q}$.
\section{Preliminaries and Background}
In this section, every time we use $\Omega$ we refer to it as a subset of some $\mathbb{R}^m$.
\begin{defi}
    An abstract simplicial complex $\Delta$ is a family of subsets of a finite set $V$, such that $\Delta$ is closed under inclusion, that is, if $\sigma \in \Delta$ and $\tau\subset\sigma$ then $\tau \in \Delta$. The word \textit{abstract} can be discarded in this manuscript. $V$ is called the vertex set of $\Delta$, and each $\sigma$ is called simplex.
\end{defi}
\par A geometric simplicial complex $\Delta$ is a finite collection of simplices in $\mathbb{R}^n$ such that:
\begin{enumerate}
    \item Every face of a simplex in $\Delta$ is also in $\Delta$.
    \item The intersection of any two simplices in $\Delta$ is either empty or a common face of both.
\end{enumerate}
The \emph{underlying space} of $\Delta$, denoted $|\Delta|$, is the union of all its simplices:
\[
|\Delta| = \bigcup_{\sigma \in \Delta} \sigma \subset \mathbb{R}^n.
\]
Intuitively, a geometric simplicial complex is a “gluing together” of points, line segments, triangles, tetrahedron, etc., in $\mathbb{R}^n$ so that simplices only meet nicely along shared faces. In fact, every $d$-dimensional abstract simplicial complex can be realized in $\mathbb{R}^{d+1}$, but this realization may fails to preserves the same topology of the corresponding abstract simplicial complex. Because of it, it is important to know that there is a Whitney's theorem \cite{whitney_1} type for simplicial complexes:
\begin{theo}
    Let $\Delta$ be a d-dimensional simplicial complex. Then, $|\Delta|$ can be embedded in $\mathbb{R}^{2 d+1}$, that is, there exists an embedding $|\Delta| \hookrightarrow \mathbb{R}^{2 d+1}$.
\end{theo}

\par Let $\Delta$ be a finite simplicial complex. The \emph{$k$-th chain group} $C_k(\Delta)$ is the free abelian group (or vector space over a chosen field) generated by the $k$-simplices of $\Delta$. The \emph{boundary operator} $\partial_k : C_k(\Delta) \to C_{k-1}(\Delta)$ maps each $k$-simplex to the alternating sum of its $(k-1)$-dimensional faces. This operator satisfies $\partial_{k-1} \circ \partial_k = 0$. A \emph{$k$-cycle} is a $k$-chain with zero boundary, forming the subgroup $Z_k(\Delta) = \ker \partial_k$. A \emph{$k$-boundary} is a $k$-chain that is the boundary of a $(k+1)$-chain, forming the subgroup $B_k(\Delta) = \operatorname{im} \partial_{k+1}$. Since every boundary is a cycle, $B_k(\Delta) \subseteq Z_k(\Delta)$. The $k$-th homology group is then defined as the quotient
\[
H_k(\Delta) = Z_k(\Delta)/B_k(\Delta),
\]
whose elements represent equivalence classes of cycles modulo boundaries. Intuitively, $H_k(\Delta)$ captures the “$k$-dimensional holes” in $\Delta$ that are not filled by higher-dimensional simplices. The rank of $H_k(\Delta)$ is the \emph{$k$-th Betti number}, $\beta_k$, which counts independent $k$-dimensional holes: $\beta_0$ corresponds to connected components, $\beta_1$ to loops, and $\beta_2$ to voids or cavities. Related to finitely generated abelian groups and homology groups we have the two following preliminaries results, which will be used later.
\begin{theo}[Structure Theorem]\label{estrutura}
    Let $G$ be finitely-generated abelian group. Then there exist nonnegative integers $r, n_1,\ldots, n_t$ such that $ G \cong \mathbb{Z}^r \oplus (\bigoplus_i \mathbb{Z}_{n_i})$,
    where $\mathbb{Z}^r$ is the free part of $G$, the sum $\bigoplus_i \mathbb{Z}_{n_i}$ is the torsion part of $G$ and $r=\operatorname{rank}G$.
\end{theo}

\begin{theo}[Universal Coefficient Theorem]\label{uct}
    If $\boldsymbol{C}$ is a chain complex of free abelian groups, then there are natural short exact sequences
    \begin{equation*}
 0 \rightarrow \boldsymbol{H}_n(\Delta) \otimes G \rightarrow \boldsymbol{H}_n(\Delta , G) \rightarrow \operatorname{Tor}\left(\boldsymbol{H}_{n-1}(\Delta), G\right) \rightarrow 0
    \end{equation*}
    for all $n$ and all $G$, and these sequences split, though not naturally.
\end{theo}
\par In this paper, we denote by $\tilde{Z_k}(\Delta)$ the set of cycles in $Z_k(\Delta)$ whose are formed by $k+2$ vertices. Note that $H_k(\Delta)$ contains classes of cycles, and each representative of these classes are independent cycles. So, not necessarily $\tilde{Z}_k(\Delta)$ will contains representatives of classes, because not all simple cycle is an independent cycle. Although, we are  interested on those with the minimal quantity of vertices. In Figure \ref{fig:cycles} we give a brief example of cycles in $\tilde{Z_k}(\Delta)$, when in that case the $1$-independent cycle if formed by the vertices $v_1, v_2, v_3$ and $v_5$.

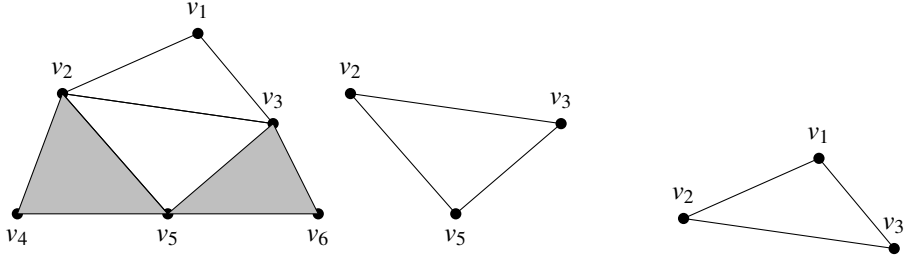
\begin{figure}[H]
    \begin{minipage}[t]{0.3\textwidth}
        \begin{tikzpicture}[scale=2]
        \node (v4) at (0,0) [circle,fill,inner sep=1.5pt,label=below:$v_4$] {};
        \node (v5) at (1,0) [circle,fill,inner sep=1.5pt,label=below:$v_5$] {};
        \node (v2) at (0.3,0.8) [circle,fill,inner sep=1.5pt,label=above:$v_2$] {};
        \node (v3) at (1.7,0.6) [circle,fill,inner sep=1.5pt,label=above:$v_3$] {};
        \node (v6) at (2,0) [circle,fill,inner sep=1.5pt,label=below:$v_6$] {};
        \node (v1) at (1.2, 1.2) [circle,fill,inner sep=1.5pt,label=above:$v_1$] {};
        
        % Fill triangles first (distinct areas)
        \fill[gray!50] (0,0) -- (1,0) -- (0.3,0.8) -- cycle;   % Triangle 1
        \fill[gray!50] (1,0) -- (1.7,0.6) -- (2,0) -- cycle;  % Triangle 2
        
        % Draw edges on top
        \draw (0,0) -- (1,0) -- (0.3,0.8) -- (0,0);
        \draw (1,0) -- (0.3,0.8) -- (1.7,0.6) -- (1,0);
        \draw (1,0) -- (2,0);
        \draw (1.7,0.6) -- (2,0);
        \draw (1.2, 1.2) -- (0.3,0.8) -- (1.7,0.6) -- (1.2, 1.2);
        \end{tikzpicture}
    \end{minipage}
    \hfill
    \begin{minipage}[t]{0.3\textwidth}
        \begin{tikzpicture}[scale=2]
        \node (v5) at (1,0) [circle,fill,inner sep=1.5pt,label=below:$v_5$] {};
        \node (v2) at (0.3,0.8) [circle,fill,inner sep=1.5pt,label=above:$v_2$] {};
        \node (v3) at (1.7,0.6) [circle,fill,inner sep=1.5pt,label=above:$v_3$] {};

        \draw (1,0) -- (0.3,0.8) -- (1.7,0.6) -- (1,0);
        \end{tikzpicture}
    \end{minipage}
        \hfill
    \begin{minipage}[t]{0.3\textwidth}
        \begin{tikzpicture}[scale=2]
        \node (v2) at (0.3,0.8) [circle,fill,inner sep=1.5pt,label=above:$v_2$] {};
        \node (v3) at (1.7,0.6) [circle,fill,inner sep=1.5pt,label=above:$v_3$] {};
        \node (v1) at (1.2, 1.2) [circle,fill,inner sep=1.5pt,label=above:$v_1$] {};
        
        \draw (1.2, 1.2) -- (0.3,0.8) -- (1.7,0.6) -- (1.2, 1.2);   
        \end{tikzpicture}
    \end{minipage}
    
   \caption{From the left to the right: A simplicial complex with two $1$-cycles. In the middle the cycle formed by the vertices $v_2, v_3$ and $ v_5$. In the very right is the other $1$-cycle with vertices $v_1, v_2$ and $ v_3$}
    \label{fig:cycles}
\end{figure}
\par In Topological Data Analysis some simplicial complexes have played an important role through the applicants of it, and also in the development of theoretical results. Namely these are are the Vietoris-Rips complex and Cech complex. As introduced before in the paper \cite{cavalcanti2025}, we will make use of the definitions of these complexes into a Finsler space.
\par Let $\Omega$ be a smooth $n$-dimensional manifold.  
A \emph{Finsler metric} on $\Omega$ is a function
\[
F : T\Omega \longrightarrow [0,\infty)
\]
such that for each $x\in\Omega$, the restriction $F_x := F(x,\cdot)$ on the tangent space $T_x\Omega$ satisfies:
\begin{enumerate}
    \item \textbf{Smoothness:} $F$ is $C^\infty$ on $T\Omega\setminus\{0\}$.
    \item \textbf{Positive homogeneity:} $F(x,\lambda v) = \lambda F(x,v)$ for all $\lambda>0$.
    \item \textbf{Strong convexity:} The square $F_x^2$ has a positive-definite Hessian with respect to $v$ on $T_x\Omega\setminus\{0\}$.
\end{enumerate}
The pair $(\Omega,F)$ is called a \emph{Finsler space}.

\begin{defi}\label{cech_complex}
    Let $X$ be a finite point set in a Finsler space $(\Omega,F)$. Then the Čech complex for $X$, attached to the parameter $\epsilon$, denoted by $\check{C}ech(X,\epsilon)$ will be the simplicial complex whose vertex set is $X$, and where $\{p_1, \ldots, p_{k+1}\}\subset X$ spans a $k$-simplex if 
    \[
    \mathcal{B}_\epsilon(p_1) \cap \cdots \cap \mathcal{B}_\epsilon(p_{k+1}) \neq \emptyset.
    \]
\end{defi}
\par There is a well-established idea in TDA that, in $\mathbb{R}^n$, taking convex balls around each sample point produces a \emph{thickened} version of that point. Viewing the entire point cloud through these balls is therefore equivalent to considering the thickened set of points. Consequently, the topology of the union of these balls reflects the presumed topology of this thickened point set. The topological invariants of this set of points are attacked in a diagram after a filtration \cite{chazal_1}.

\begin{defi}
    Let $X$ be a finite point set in a Finsler space $(\Omega,F)$. The Vietoris-Rips complex for $X$, attached to the parameter $\epsilon$, denoted by $Rips(X,\epsilon)$ will be the simplicial complex whose vertex set is $X$, and where $\{p_1, \ldots, p_{k+1}\}\subset X$ spans a $k$-simplex if 
    \[
    d_{p_i}(p_i, p_j)\leq \epsilon \quad \text{for all} \quad p_i, p_j \in \left\{p_1, \ldots, p_{k+1}\right\}.
    \]
\end{defi}

\begin{prop}
    Let $\Delta$ be a simplicial complex, and suppose $\tilde{Z_k}(\Delta)\neq \emptyset$, for some $k<d$. Then, $\Delta$ can not be the $Vietoris-Rips$ complex of any metric space neither any Finsler space.
\end{prop}
\begin{proof}
    Suppose $\tilde{Z_k}(\Delta)\neq \emptyset$ and $\Delta = Rips(X,\epsilon)$ for a distance $d$ or a Finsler metric $F$. If $V = \{p_0, p_1, \ldots, p_{k+1}\}$ is the vertex set for a $k$ cycle in $\tilde{Z_k}(\Delta)$, then there is no $\Delta_{k+1}$ simplex on these vertices. Since $\Delta$ is a $Rips$ complex, then for $p_i \in V$,  $d(p_i, p_j)\leq \epsilon$, in the metric space case, or $d_{p_i}\left(p_i, p_j\right) = F_{p_i}(p_j-p_i) \leq \varepsilon$ in a Finsler space, which build a $\Delta_{k+1}$ simplex. This gives a contradiction.
\end{proof}
\par Let $\Bbbk$ be a fix ring and $I$ an interval of $\mathbb{R}$.
\begin{defi}
    A filtration of a finite set $X$ is a family $\{\Delta_\epsilon\}_{\epsilon\in I}$ of subcomplexes of the complete graph with vertices on $X$, such that $\Delta_\epsilon\subseteq\Delta_{\epsilon^{\prime}}$, whenever $\epsilon \leq \epsilon^{\prime}$.
\end{defi}
\par For each inclusion $\Delta_\epsilon\hookrightarrow\Delta_{\epsilon^{\prime}}$, there is the induced linear map $v_{\epsilon}^{\epsilon^{\prime}}:H_k(\Delta_\epsilon;\Bbbk)\to H_k(\Delta_{\epsilon^{\prime}};\Bbbk)$, for all $\epsilon \leq \epsilon^{\prime}$. Plus, one may see that $v_{\epsilon}^{ \varepsilon}=\operatorname{id}$, and $v_{\epsilon^{\prime}}^{\epsilon^{\prime \prime}}\circ v_{\epsilon}^{\epsilon^{ \prime}}=v_{\epsilon}^{\epsilon^{\prime \prime}}$. Following the notation and backgrounds in \cite{chazal2016structure}, this defines precisely a \textbf{Persistence Module} $\mathbb{V}$ over $I$, which may be seen categorically as a functor $\mathbb{V}:(I, \leq) \longrightarrow \textbf{Vect}_{\Bbbk}$.

\section{Topological Obstructions in Finsler Spaces}
\par Let $\Delta$ be a finite simplicial complex. We say that $\Delta$ is $d$-representable if there exists a family of convex sets $\mathcal{U}=\{U_1, \ldots, U_k\}$ in $\mathbb{R}^d$ such that $\mathcal{N}(\mathcal{U})$ is isomorphic to $\Delta$. For instance, the nerve of a convex finite family of sets in $\mathbb{R}^d$ is itself $d$-representable. In particular, a Čech complex in $\mathbb{R}^d$ given by a Finsler space, as in Definition \ref{cech_complex}, satisfies the $d$-representability condition, for any $d$.

\par Representability offers a natural framework for advancing TDA, by clarifying the distinction between topological features that admit geometric realization and those that are merely combinatorial. For instance, the following theorem makes it clear.
\begin{theo}[Perel'man, \cite{perelman_1} and Wegner, \cite{wegner_1}]
    Let $\Delta$ be a d-dimensional simplicial complex. Then $\Delta$ is (2d+1)-representable.
\end{theo}

Although it is now known that \(2d+1\) is the sharp lower bound for the representability of a \(d\)-dimensional complex, this question remained open for more than forty years after Wegner's theorem. In 2011, Tancer finally resolved it, proving the following statement \cite{martin_1}.

\begin{theo}[Tancer’s Tightness]
For every $d>0$ there exists a $d$-dimensional simplicial complex that is not $2d$-representable, so the $2d+1$ bound is optimal.
\end{theo}
Most current TDA implementations do not incorporate embeddings into spaces that reflect the underlying geometry of the data, and therefore may fail to capture its true topological features. In the following discussion we use the Finsler metric to avoid that phenomenon.
\begin{defi}
    We say that a simplicial complex $\Delta$ captures the topology of $X$ if $\Delta$ is the $\check{C}ech$ complex, $\check{C}ech(X,\epsilon)$, and the nerve theorem applies.
\end{defi}
\begin{prop}\label{obstruction}
    Let $\Delta$ be a simplicial complex with vertex set $X$ in $\mathbb{R}^n$ and suppose $\tilde{Z_k}(\Delta) \neq \emptyset$.  If $k \geq n+1$, then $\Delta$ does not capture the topology of $X$ in a Finsler space.
\end{prop}
\begin{proof}
    Let $\Delta$ be as in the hypothesis. Suppose that $\Delta$ capture the topology of $X$, as a data set in some Finsler space with metric $F$. By the Nerve theorem, we shall assume that $\Delta=\check{C} e \operatorname{ch}(X, \varepsilon)$, for some $\epsilon>0$. Writing $\mathcal{B}_{\varepsilon}\left(p_i\right)$ as the ball of radius $\epsilon$, centered at $p_i \in V$, take a $k$-cycle $z \in\tilde{Z_k}(\Delta)$, lets say the cycle with vertices $p_0, p_1,\ldots,p_{k+1}$ and let $\mathcal{C}$ be the family of balls, given by $F$, of radius $\epsilon$ and centered at these points. Because $F_{p_i}$ invariably produces convex balls, then $\mathcal{C}$ is a convex family.   
    \begin{claim}
        Any $n+1$ balls of $\mathcal{C}$ has non-empty intersection.
    \end{claim}
    By definition, there are $k+2$ simplices of dimension $k$ being glued to form $z$. Any $k$ simplices of these $k+2$ simplices share a $(k-1)$-simplex. Then, let $p_{i_0}, p_{i_1}, \ldots, p_{i_k}$ be the vertices of this $(k-1)$-simplex, and by construction of $\check{C} e \operatorname{ch}(X, \varepsilon)$ we have
    \[
    \bigcap_{j=0}^{k} \mathcal{B}_{\varepsilon}\left(p_{i_j}\right) \neq \emptyset.
    \]
\par Hence, any choice of $k \geq n+1$ balls in $\mathcal{C}$ satisfies the Helly's theorem condition in $\mathbb{R}^n$. Therefore, $\bigcap_{C \in \mathcal{C}} C \neq \emptyset$. And by definition, $\{p_0, p_1, \ldots, p_{k+1}\}$ builds a $(k+1)$-simplex of $\Delta$ instead of a $k$-cycle. Thus, $\Delta$ can not be $\check{C} e \operatorname{ch}(X, \varepsilon)$, and consequently can not capture the topology of $X$ in a Finsler space with the metric $F$.
\end{proof}
\begin{coro}\label{geometricdegree}
Let $X \subset \mathbb{R}^n$ be a finite set and let 
$\{\Delta_\varepsilon\}_{\varepsilon>0}$ be the \v{C}ech filtration induced by a Finsler metric.
Then, for any field $\Bbbk$, we have
\[
H_k(\Delta_\varepsilon;\Bbbk)=0
\quad\text{for all }\varepsilon>0\text{ and all }k>n.
\]
In particular, the $k$-th persistent homology module is trivial for all $k>n$.
\end{coro}
\begin{proof}
    Let $\mathbb{V}$ be the persistence module $H_k\left(\Delta_{\bullet} ; \Bbbk\right)$. The result follows immediately from Proposition \ref{obstruction}.
\end{proof}

\section{Coefficient on Persistence Homology Module}
\par Before going specially into the main objects whose will be discussed in this section, we may present a fantastic decomposition made by the great mathematician \textit{Henry Smith} in his work \cite{smithform}.
\begin{theo}\label{sfn}
    Suppose $R$ is a PID and let $A$ be a nonzero $m \times n$ matrix over $R$. Then there exist invertible matrices $m \times m$ and $n \times n$ matrices $U,V$ over $R$ such that the product $UAV$ is 
    \[
\left(
\begin{array}{c|c}
\begin{matrix}
\alpha_1 & & 0 \\
& \ddots & \\
0 & & \alpha_r
\end{matrix}
& \boldsymbol{0}_{r\times n-r} \\ \hline
\boldsymbol{0}_{m-r \times r} & \boldsymbol{0}_{m-r \times n-r }
\end{array}
\right)=\operatorname{diag}(\alpha_1, \ldots, \alpha_r,0 \ldots, 0) 
    \]
where each $\alpha_i$ in the diagonal satisfies $\alpha_i \mid \alpha_{i+1}$, for all $1 \leq i<r$. These values are determined by $\alpha_i=\displaystyle\frac{d_i(A)}{d_{i-1}(A)}$, once known that $d_i(A)$ is the greatest common divisor of the determinants of all $i\times i$ minors of $A$. The values $\alpha_i$ are called the \textbf{elementary divisors} of $A$. By convention we set $d_0(A)=1$.
\end{theo}
\par We say that such diagonal matrix in Theorem \ref{sfn} is the \textbf{Smith normal Form} or simply SNF of the matrix $A$. Then, particularly, for all integer matrix $A$ one can calculate the Smith normal form of $A$.

\par For each boundary map $\boldsymbol{\partial}_k: \mathbf{C}_k \longrightarrow \mathbf{C}_{k-1}$, let us skip the use of the subscript $k$ in $\partial_k$, when the chain complexes are explicit shown. For the further aim of this chapter, consider the following diagram. 
\[
\begin{array}{ccccccc}
\boldsymbol{0}\longrightarrow{} \cdots & \boldsymbol{C}_{k+2}\xrightarrow{\boldsymbol{\partial}} & \boldsymbol{C}_{k+1} & \xrightarrow{\boldsymbol{\partial}} & \boldsymbol{C}_{k} & \xrightarrow{\boldsymbol{\partial}} & \cdots \longrightarrow{} \boldsymbol{0}\\
& \big\downarrow \boldsymbol{\partial}& \big\downarrow \boldsymbol{\partial} &  & \big\downarrow \boldsymbol{\partial} & & \\
\boldsymbol{0}\longrightarrow{}\cdots &\boldsymbol{C}_{k+1} \xrightarrow{\boldsymbol{\partial}} & \boldsymbol{C}_k & \xrightarrow{\boldsymbol{\partial}} & \boldsymbol{C}_{k-1} & \xrightarrow{\boldsymbol{\partial}} & \cdots \longrightarrow{} \boldsymbol{0}
\end{array}
\]
\par If each $\boldsymbol{C}_{k}(\Delta, \mathbb{Z)}$ is spanned by $m_k$ $k$-simplices, we may do the identification $\boldsymbol{C}_{k}\cong \mathbb{Z}^{m_k}$. Letting $A_{m_k}$ denote the matrix $[\boldsymbol{\partial}]$ of $\boldsymbol{\partial}: \mathbf{C}_k \longrightarrow \mathbf{C}_{k-1}$ in terms of generators simplices of $\mathbf{C}_k$ and $\mathbf{C}_{k-1}$. Thence, there is a natural identification of the boundary map as $A_{m_k}:\mathbb{Z}^{m_k}\longrightarrow \mathbb{Z}^{m_{k-1}}$. Since the entries of $A_{m_k}$ are taken from the set $\{-1,0,1\}$, we may decompose $A^{m_k}$ in the SNF. Now, we give two established results as lemmas for our proposes.
\begin{lema}
    If $\alpha_1, \alpha_2,\ldots, \alpha_r$ are the elementary divisors of an $m\times n$ integer matrix  $A$. Then, the $\operatorname{co-kernel}$ of $A:\mathbb{Z}^m \longrightarrow\mathbb{Z}^n$ satisfies $\operatorname{coker}(A):=\mathbb{Z}^n / \operatorname{im}(A) \cong \bigoplus_{i=1}^r \mathbb{Z}_{\alpha_i} \oplus \mathbb{Z}^{n-r}$.
\end{lema}
\begin{proof}
    Note that $\operatorname{coker}(A) \cong \operatorname{coker}(B)$ via $x+\operatorname{im}(A) \mapsto U^{-1} x+\operatorname{im}(B)$, for $A=U B V$.
\end{proof}
\par An important case is the description of a homology group over $\mathbb{Z}$ using the previous lemma.
\begin{lema}\label{lem_homo}
Let $B:\mathbb{Z}^m \longrightarrow\mathbb{Z}^n$ and $A:\mathbb{Z}^n \longrightarrow\mathbb{Z}^d$ be integers matrices such that $AB=0$. If $r$ is the $\operatorname{rank}$ of $B$ and $s$ is the $\operatorname{rank}$ of $A$, then $\operatorname{ker}(A) / \operatorname{im}(B) \cong \bigoplus_{i=1}^r \mathbb{Z}_{\alpha_{i}} \oplus \mathbb{Z}^{n-r-s}$. 
\end{lema}
\begin{proof}
    The proof consists in consider the homomorphism $\bar{A}: \mathbb{Z}^n / \operatorname{im}(B) \rightarrow \mathbb{Z}^l$ and noticing that $\operatorname{ker}(A) / \operatorname{im}(B)\cong \operatorname{ker}(\bar{A})$. The previous lemma ensures $\mathbb{Z}^n / \operatorname{im}(B) \cong \bigoplus_{i=1}^r \mathbb{Z}_{\alpha_i} \oplus \mathbb{Z}^{n-r}$. Since we define $\bar{A}(\bar{z})=A(z)$, for all classes $\bar{z}\in \mathbb{Z}^n / \operatorname{im}(B)$, one may note that $\operatorname{im}(A)=\operatorname{im}(\bar{A})$. Because $\operatorname{im}(\bar{A})$ is free, the torsion of $\mathbb{Z}^n / \operatorname{im}(B)$ must be in the kernel of $\bar{A}$. Therefore, $\operatorname{ker}(\bar{A})=\bigoplus_{i=1}^r \mathbb{Z}_{\alpha_i}\oplus \operatorname{ker}\left(\left.\bar{A}\right|_{\mathbb{Z}^{n-r}} \right)$, which implies that $\operatorname{im}(\bar{A})=\operatorname{im}\left(\left.\bar{A}\right|_{\mathbb{Z}^{n-r}}\right)$. According to the equation $n-r=\operatorname{ker}(A)+s$, it is possible to obtain the desired isomorphism. 
\end{proof}
\begin{prop}\label{notorsion}
    Let $\alpha_1,\ldots,\alpha_r$ and $\tilde{\alpha}_1,\ldots,\tilde{\alpha}_s$ be the elementary divisors of $\boldsymbol{\partial}_{k+1}$ and $\boldsymbol{\partial}_{k}$, and $p$ any prime such that $\operatorname{gcd}\left(p, \alpha_r \tilde{\alpha}_s\right)=1$. Then, $\boldsymbol{H}_k\left(\Delta; \mathbb{Z}_p\right)\cong \mathbb{Z}_p^{\beta_k}$.
\end{prop}
\begin{proof}
    Let $\boldsymbol{H}_k(\Delta)$ denotes the $kth$-homology group of $\Delta$ over $\mathbb{Z}$ and $\beta_k$ its Betti number. By the Structure theorem \ref{estrutura}, there are some integers $t_1, \ldots,t_r$ such that $\boldsymbol{H}_k(\Delta)\cong \bigoplus_{i=1}^r \mathbb{Z}_{t_i} \oplus\mathbb{Z}^{\beta_k}$. Now, using the Universal Coefficient theorem \ref{uct} for $G=\mathbb{Z}_p$ one has
\begin{align*}
\boldsymbol{H}_k(\Delta; \mathbb{Z}_p) &\cong (\boldsymbol{H}_k(\Delta) \otimes \mathbb{Z}_p) \bigoplus \operatorname{Tor}(\boldsymbol{H}_{k-1}(\Delta), \mathbb{Z}_p) \\
  &\cong \left(\left(\bigoplus_{i=1}^r \mathbb{Z}_{t_i} \oplus \mathbb{Z}^{\beta_k}\right)\otimes \mathbb{Z}_p \right)\bigoplus \operatorname{Tor}(\boldsymbol{H}_{k-1}(\Delta), \mathbb{Z}_p) \\
  &\cong \left(\left(\bigoplus_{i=1}^r (\mathbb{Z}_{t_i}\otimes\mathbb{Z}_p)  \right)
  \oplus \mathbb{Z}_p^{\beta_k}\right)\bigoplus \operatorname{Tor}\left(\boldsymbol{H}_{k-1}(\Delta), \mathbb{Z}_p\right),
\end{align*}
now apply the Structure theorem for $\boldsymbol{H}_{k-1}(\Delta)$, which gives $\boldsymbol{H}_{k-1}(\Delta) \cong \bigoplus_{i=1}^s \mathbb{Z}_{\tilde{t}_i} \oplus \mathbb{Z}^{\beta_{k-1}}$. Consequently, $\operatorname{Tor}\left(\boldsymbol{H}_{k-1}(\Delta), \mathbb{Z}_p\right)  \cong \operatorname{Tor}\left(\bigoplus_{i=1}^s \mathbb{Z}_{\tilde{t}_i} \oplus \mathbb{Z}^{\beta_{k-1}}, \mathbb{Z}_p\right) \cong\bigoplus_{i=1}^s (\mathbb{Z}_{\tilde{t}_i}\otimes\mathbb{Z}_p)$. One property of $\operatorname{Tor}$ is that $\operatorname{Tor}\left(\mathbb{Z}_n, \mathbb{Z}_m\right) \cong \mathbb{Z}_{\operatorname{gcd}(n, m)}$.
\par On the other hand, let us identify $\boldsymbol{C}_k(\Delta, \mathbb{Z}) \cong \mathbb{Z}^{m_k}$ and let $A_{m_k}$ be the matrix $[\boldsymbol{\partial}_k]$. Now, consider the following identified sequence 
\[
\mathbb{Z}^{m_{k+1}} \xrightarrow{A_{m_{k+1}}} \mathbb{Z}^{m_{k}} \xrightarrow{A_{m_k}} \mathbb{Z}^{m_{k-1}} \xrightarrow{A_{m_{k-1}}} \mathbb{Z}^{m_{k-2}}
\]
where $\operatorname{rank} A_{m_{k}}=r_k$. According to Lemma \ref{lem_homo}, $\boldsymbol{H}_k(\Delta)$ and $\boldsymbol{H}_{k-1}(\Delta)$ may both be written as $\boldsymbol{H}_k(\Delta) \cong \bigoplus_{i=1}^r \mathbb{Z}_{\alpha_i} \oplus \mathbb{Z}^{m_k-r_k-r_{k-1}}$ and $\boldsymbol{H}_{k-1}(\Delta) \cong \bigoplus_{i=1}^s \mathbb{Z}_{\tilde{\alpha}_i} \oplus \mathbb{Z}^{m_{k-1}-r_{k-1}-r_{k-2}}$. Relating with the isomorphisms given in the begging of the proof, we must have 
\[
m_k - r_k - r_{k-1} = \beta_k, \quad 
\prod_{i=1}^r \alpha_i = \prod_{i=1}^r t_i, \quad 
m_{k-1} - r_{k-1} - r_{k-2} = \beta_{k-1}, \quad 
\prod_{i=1}^s \tilde{\alpha}_i = \prod_{i=1}^s \tilde{t}_i.
\]
Let $A_j=\{\text{the set of all $j\times j$ minors of $A_{m_{k+1}}$}\}$ and $\tilde{A}_j=\{\text{the set of all $j\times j$ minors of $A_{m_k}$}\}$. These sets define the coefficients $\alpha_i=\frac{\operatorname{gcd}\{A_i\}}{\operatorname{gcd}\{A_{i-1}\}}$ and $\tilde{\alpha}_i = \frac{\operatorname{gcd}\{\tilde{A}_i\}}{\operatorname{gcd}\{\tilde{A}_{i-1}\}}$. Our focus is on $\alpha_r$ and $\tilde{\alpha}_s$. If $d$ is such a prime that does not divides $\alpha_r$ neither $\tilde{\alpha}_s$ we have $\operatorname{gcd}(d,\alpha_r\tilde{\alpha}_s)=1$, and also the converse is true. Take $p$ as in the hypothesis, then $\operatorname{gcd}(p,t_r)=1$ and $\operatorname{gcd}(p,\tilde{t}_s)=1$, because as $p$ does not divides $\alpha_r$ neither $\tilde{\alpha}_s$, the same holds for $\prod_{i=1}^r t_i$ and $\prod_{i=1}^s \tilde{t}_i$.
\end{proof}
\begin{coro}\label{goodprimes}
Let $\{\Delta_\epsilon\}_{\epsilon \in I}$ be a finite filtration.
Then there exists a finite set $P$ of primes such that, for all $p\notin P$
and all $\epsilon\in I$,
\[
\boldsymbol{H}_k(\Delta_\epsilon;\mathbb{Z}_p)
\cong
\mathbb{Z}_p^{\beta_k(\Delta_\epsilon)}.
\]
\end{coro}
\begin{proof}
    For each $\epsilon$, Proposition \ref{notorsion} the existence of a finite set $P_\epsilon$ built of primes dividing the elementary divisors of $\boldsymbol{\partial}_{k+1}$ neither $\boldsymbol{\partial}_{k}$. The filtration is essentially determined by $N+1$ stages: $\epsilon_0, \epsilon_1,\ldots, \epsilon_N$. Now, just consider $P=\bigcup_{i=0}^N P_{\epsilon_i}$.
\end{proof}

\begin{remark}
Since $\mathbb{Q}$ and $\mathbb{Z}_p$ are fields, persistent homology with coefficients
in either field admits an interval decomposition.
By the previous corollary, for all but finitely many primes $p$ we have
\[
\boldsymbol{H}_k(\Delta_\epsilon;\mathbb{Z}_p)\cong \mathbb{Z}_p^{\beta_k(\Delta_\epsilon)}
\quad \text{for all } \epsilon.
\]
Moreover, for all $\epsilon$,
\[
\dim_{\mathbb{Q}} \boldsymbol{H}_k(\Delta_\epsilon;\mathbb{Q})
=\beta_k(\Delta_\epsilon)
=\dim_{\mathbb{Z}_p} \boldsymbol{H}_k(\Delta_\epsilon;\mathbb{Z}_p).
\]
Consequently, for every $\epsilon \le \epsilon'$,
\[
\operatorname{rank}\!\left(
\boldsymbol{H}_k(\Delta_\epsilon \to \Delta_{\epsilon'};\mathbb{Z}_p)
\right)
=
\operatorname{rank}\!\left(
\boldsymbol{H}_k(\Delta_\epsilon \to \Delta_{\epsilon'};\mathbb{Q})
\right).
\]
Thus, the rank invariants of the persistence modules
$\boldsymbol{H}_k(\Delta_\bullet;\mathbb{Q})$ and $\boldsymbol{H}_k(\Delta_\bullet;\mathbb{Z}_p)$ coincide.
By the classification theorem for pointwise finite-dimensional persistence modules
over a field, the rank invariant uniquely determines the interval decomposition.
Therefore, persistent homology over $\mathbb{Q}$ and over $\mathbb{Z}_p$ has the same
intervals with the same multiplicities; see, for instance,
\cite{crawleyboevey2015,chazal2016structure}.
\end{remark}

\par The study of homology of simplicial complexes takes a special place when computing it to the Rips complex due to its use on a filtration of a set of points and its diagram persistence. Even we have defined the Vitoris-Rips complex for a finite subset of a metric or Finsler space, it is completely fine to extend this definition and defining it on an infinite set of points. One of the main manuscripts on the studying of the Rips complex of a continuous manifold is the paper, where the author gives enough condition to study the homology of a manifold through a Rips complex and reciprocally, as you can see in the following theorem.
\begin{theo}[Hausmann, \cite{Hausmann1996}]
    Let $M$ be a closed Riemannian manifold, then there exists a sufficiently small $\epsilon$ such that $Rips(M,\epsilon)$ is homotopy equivalent to $M$. 
\end{theo}
\par One of the questions left by the author in that paper is the conjecture about the possibility of recovering the homology of $M$ by the Rips complex of a finite subset of $M$ when adding compactness on $M$. Recently, an affirmative answer for the circle $\mathbb{S}^1$ was given in \cite{Adama}; we reproduce their stronger version theorem \textit{ipsissimis litteris}:
\begin{theo}
If $X$ is dense in $\mathbb{S}^{1}$ (in particular when $X = \mathbb{S}^{1}$) and $0 < r < \tfrac12$, then we have
\[
Rips(X, r) \simeq \mathbb{S}^{2l+1}
\qquad\text{for}\qquad
\frac{l}{2l+1} < r < \frac{l+1}{2l+3},
\quad l = 0,1,2,\ldots
\]
\end{theo}
\begin{remark}
    Even this theorem states the homology for dense subsets of $\mathbb{S}^1$, in the same paper they prove a similar result when $X$ is a finite subset of the circle.
\end{remark} 
\par The previous theorem and its related paper answered the question of Hausmann for an example of compact manifold and a finite Rips complex. However, in the same direction, Letschev has proved in the early years of 2000s, that for a closed Riemannian manifold $M$, a dense subset of $M$, under some conditions, has the same homology as $M$. See the following theorem to understand.
\begin{theo}[Latschev, \cite{Latschev}]
    Let $M$ be a closed Riemannian manifold. Then there exists $\epsilon_0>0$ such that for every $0<\epsilon\leq\epsilon_0$ there exists a $\delta>0$ such that the geometric realization of $Rips(Y,\epsilon)$ of any metric space $Y$ which has Gromov-Hausdorf distance less than $\delta$ to $M$ is homotopy equivalent to $M$.
\end{theo}
\par We skip the proof of these last three theorems, because we use them to contextualize and gives to us examples of Rips complex of a finite set having torsion.
\begin{ex}
    Let $M=\mathbb{RP}^2$. We already know that $\boldsymbol{H}_1(M,\mathbb{Z)}\cong\mathbb{Z}_2$, and then by Hausmann theorem, $Rips(M,\epsilon)\simeq M$, for sufficiently small $\epsilon$. Let us construct a subset $X$ of $M$ such that $Rips(X,\epsilon)$ is still homotopy equivalent to $M$. Let $C_0$ be the semicircle of $\mathbb{S}^2$ in the $xy$ plane. If $\delta$ is any positive real number, let $N=\lceil \frac{\pi}{\delta} \rceil$ points be evenly spaced around the semicircle at angles $\theta_k = \frac{\pi k}{N}$, $k=0,1,\ldots,N-1$. For each of these $N$ points on $C_0$, take the semicircle connecting this point to its antipodal, by $z\geq0$. Hence, we obtain $N+1$ semicircles, $C_0, C_1, \ldots, C_N$:
    \[
\begin{array}{c|cccc}
\text{points of semicircle } C_0 & p_{00} & p_{01} & \cdots & p_{0N} \\
\text{points of semicircle } C_1 & p_{10} & p_{11} & \cdots & p_{1N} \\
\vdots        & \vdots & \vdots & \ddots & \vdots \\
\text{points of semicircle } C_N & p_{N0} & p_{N1} & \cdots & p_{NN}
\end{array}.
\]
\par Now, lets prove that given a point $p$ in the upper hemisphere of $\mathbb{S}^2$, there exists some $p_{ij}$ such that $d_{\mathbb{S}^2}(p,p_{ij})\leq\delta$. If $p$ lies in any of the semicircles, it is done, by construction. But if $p$ does not belongs to any of the semicircles, we may find $C_i$ and $C_{i+1}$ such as $p$ lies between these two curves. Consider another semicircle $\tilde{C}$ passing through $p$, then there exist $p_{ik}$ and $p_{(i+1)k}$ such that, $p$ has distance less than $\frac{\delta}{2}$. Applying the triangle inequality we obtain $d_{\mathbb{S}^2}(p,p_{(i+1)k})\leq d_{\mathbb{S}^2}(p,\tilde{p}_{ik}) + d_{\mathbb{S}^2}(\tilde{p}_{ik},p_{(i+1)k})\leq \delta$, where $\tilde{p}_{ik}$ is a point given by the intersection of $\tilde{C}$ and the arc connection $p_{ik}$ and $p_{(i+1)k}$. Hence, the set $S=\{p_{ij}\}$ is $\delta$-dense in $\mathbb{S}^2$. We may identify $M$ as $\mathbb{S}^2$ with the antipodal relation. Given $\bar{x},\bar{y}\in M$, we compute their distance by $d_{M}(\bar{x},\bar{y})=\operatorname{min}\{d_{\mathbb{S}^2}(x,y), d_{\mathbb{S}^2}(x,-y)\}$, and then $X=\{\bar{p}_{ij}\}$ is $\delta$-dense in $M$. One basic fact is that the Gromov-Hausdorf distance is bounded by the Hausdorf distance $d_{GH}^M(A,B)\leq d_H^M(A,B)$, where $d_H^M(A,B) = \displaystyle\max\left\{\sup_{a\in A} d_M(a,B),\;
\sup_{b\in B} d_M(b,A) \right\},$ for $A$ and $B$ subsets of $M$. Therefore, for $X$ and $M$ we have $d_{GH}^M(X,M)\leq d_H^M(X,M)\leq \delta$. Since $M$ satisfies the Latschev theorem hypothesis, there exists $\epsilon$ in such way we may choose $\delta$ in our construction as the one in the theorem in order to obtain the homotopy equivalence between $M$ and $Rips(X,\epsilon)$.
\end{ex}
\begin{ex}
    View the Klein bottle $K$ as the quotient of the unit square $[0,1]^2$ by the identifications $(x,0) \sim (x,1)$, for all $x$, and $(0,y) \sim (1,1-y)$ for all $y$,
and write $\pi : [0,1]^2 \to K$ for the quotient map.
Equip $[0,1]^2$ with the Euclidean metric and $K$ with the induced quotient metric. A construction of a $\delta$-dense subset of $K$ is also possible, and since $K$ satisfies the Latschev theorem, we may have the Rips complex of a finite set of points of $K$ homotopy equivalent to $K$, which has $\boldsymbol{H}_1(K,\mathbb{Z)}\cong \mathbb{Z}\oplus\mathbb{Z}_2$. 
\end{ex}
\par It is known that torsion may appears in generic simplicial complex, for example, the triangulation of $\mathbb{RP}^2$ is one with torsion in the first homology group, and as we saw above there exist Rips complexes with torsion as well. Since computational topological data analysis does filtration based on Rips complexes for different scales, we must know that possibly torsion may appear during the process. What Proposition \ref{notorsion} proposes is a scalar field such as the torsional cycles are excluded during filtration.

\end{document}